\documentclass[11pt,leqno]{amsart}
\usepackage{times}
\newtheorem{thm}{Theorem}
\newtheorem{lem}[thm]{Lemma}
\newtheorem{cor}[thm]{Corollary}

\begin{document}

\title[Gaussian integral means of entire functions: logarithmic convexity and concavity]
{\bf Gaussian integral means of entire functions: logarithmic convexity and concavity }
\thanks{Jie Xiao is supported in part by NSERC of Canada and URP of Memorial University}

\author{Chunjie Wang}
\address{Chunjie Wang, Department of Mathematics, Hebei University of Technology,
Tianjin 300401, China}
\email{wcj@hebut.edu.cn}

\author{Jie Xiao}
\address{Jie Xiao, Department of Mathematics and Statistics, Memorial University, St. John's, NL A1C 5S7, Canada}
\email{jxiao@mun.ca}

\begin{abstract}
For $0<p<\infty$ and $\alpha\in (-\infty,\infty)$ we determine when the $L^p$ integral mean on $\{z\in\mathbb C: |z|\le r\}$ of an entire function with respect to the Gaussian area measure
$e^{-\alpha|z|^2}\,dA(z)$ is   logarithmic  convex or   logarithmic  concave.
\end{abstract}



\maketitle

\section{Introduction}\label{s1}

Let $dA$ be the Euclidean area measure in the finite complex plane $\mathbb{C}$. For any real number $\alpha$ and $0<p<\infty$,
the Gaussian integral means of an entire function $f: \mathbb C\to\mathbb C$ are defined by
$$
 \mathsf{M}_{p,\alpha}(f,r)=\frac{\displaystyle\int_{\{z\in\mathbb C:\ |z|\le r\}}|f(z)|^pe^{-\alpha|z|^2}\,dA(z)}
{\displaystyle\int_{\{z\in\mathbb C:\ |z|\le r\}}e^{-\alpha|z|^2}\,dA(z)},\quad 0<r<\infty.
$$
This concept lies in the theory of Fock spaces; see \cite{Z}. 

The famous Hadamard's three circles theorem for the above entire function $f$ (cf. \cite{R}) states that if 
$$
\begin{cases}
0<r_1<r<r_2<\infty;\\
{\mathsf M}(f,s)=\max\{|f(z)|;\ |z|\le s\}\quad\hbox{for}\quad s\in (0,\infty),
\end{cases}
$$
then
$$
\Big(\ln\frac{r_2}{r_1}\Big)\ln\mathsf{M}(f,r)\le\Big(\ln\frac{r_2}{r}\Big)\ln\mathsf{M}(f,r_1)+\Big(\ln\frac{r}{r_1}\Big)\ln\mathsf{M}(f,r_2),
$$
i.e., $\ln\mathsf{M}(f,r)$ is convex in $\ln r$. Continuing from \cite{WX} and its prior work \cite{WXZ, WZ, XX, XZ}, this paper investigates such an analogous problem: {\it When is the function
$r\mapsto  \ln\mathsf{M}_{p,\alpha}(f,r)$ convex or concave in $ \ln r$?} In what follows, we will see that a resolution of this question depends on the parameter $\alpha$ and its induced function
$$
\varphi(x)=\frac{1-e^{-\alpha x}}{\alpha},\quad 0<x<\infty.
$$

\begin{thm}\label{t1} Let $\alpha<0$. Suppose both $x\mapsto M(x)$ and $x\mapsto M'(x)$ are positive on $(0,\infty)$. Then the function
$$x\mapsto  \ln\frac{\displaystyle\int_0^xM(t)e^{-\alpha t} dt}
{\displaystyle\int_0^xe^{-\alpha t} dt}$$
is convex in $ \ln x$ for $x$ in an open interval $I\subset(0,\infty)$ provided that the following conditions are satisfied:
\begin{enumerate}
\item[(i)] $x\mapsto \ln M(x)$ is convex in $ \ln x$ for $x\in I$;
\item[(ii)] $$x\frac{M'(x)}{M(x)}\ge\frac{[x(1-\alpha x)-\varphi(x)][\alpha\varphi(x)^2-2(1+\alpha x)\varphi(x)+2x]}{(\varphi(x)-x)[x-(1+\alpha x)\varphi(x)]}\ \ \hbox{for}\ \ x\in I.$$
\end{enumerate}
\label{thm3}
\end{thm}

As a straightforward consequence of Theorem \ref{t1}, the following   logarithmic  convexity for $\mathsf{M}_{p,\alpha}(f,\cdot)$ is similar to Corollary 8 in \cite{WX}.
\begin{cor}\label{c1}
Let $(\alpha,p)\in (-\infty,0)\times(0,\infty)$. If $f:\mathbb C\to\mathbb C$ is an entire function, then $r\mapsto \ln  \mathsf{M}_{p,\alpha}(f,r)$
is convex in $ \ln r$ for $r\in\left(0, \sqrt{\frac{t_0}{-\alpha}}\right)$, where $t_0=1.79\cdots$ is the unique root of $u(t)=e^{t}-1-t-t^2$ on $(0,\infty)$.
\end{cor}

During the process of extending Theorem \ref{t1} from $I$ to $(0,\infty)$, we find the following assertion.

\begin{thm}\label{t2} Let $\alpha<0$. Suppose both $x\mapsto M(x)$ and $x\mapsto M'(x)$, are positive on $(0,\infty)$. Then the function
$$x\mapsto  \ln\frac{\displaystyle\int_0^xM(t)e^{-\alpha t}dt}
{\displaystyle\int_0^xe^{-\alpha t} dt}$$
is convex in $ \ln x$ for $x\in(0,\infty)$ provided that $$
\left(x\frac{M'(x)}{M(x)}\right)'\ge\begin{cases}
0,\qquad x\in (0,x_0);\\
 \frac{[x(1-\alpha x)-\varphi(x)]^2}{4x(\varphi(x)-x)^2}, \qquad x\in[x_0,\infty),
\end{cases}
$$
where $x_0=-t_0/\alpha$ is the unique root of $x(1-\alpha x)-\varphi(x)$ on $(0,\infty)$.
\end{thm}

As a by-product of Theorem \ref{t2}, the following corollary extends the   logarithmic  convexity of $\mathsf{M}_{p,\alpha}(f,\cdot)$ from $I$ to $(0,\infty)$.

\begin{cor}\label{c2}
Let $(\alpha,p)\in (-\infty,0)\times(0,\infty)$. If $f:\mathbb C\to\mathbb C$ is an entire function, then $r\mapsto \ln \mathsf{M}_{p,\alpha}(f,r)$
is convex in $ \ln r$ for $r\in(0,\infty)$ if
$$\left(x\frac{M'(x)}{M(x)}\right)'\ge \frac{[x(1-\alpha x)-\varphi(x)]^2}{4x[\varphi(x)-x]^2}, \qquad x\ge x_0,$$where $x=r^2, M(x)=\int_0^{2\pi}|f(\sqrt{x}e^{i\theta})|^p\,d\theta$ and $x_0$ is as the same as above.
\end{cor}

However, whenever handling the   logarithmic  concavity we have only one situation as follows.

\begin{thm}\label{t3} Let $\alpha\ge 0$. Suppose $M(x)\ \&\ M'(x)$ are positive and $M''(x)$ exists for $x\in(0,\infty)$. If $x\mapsto \ln M(x)$ is concave in $ \ln x$ for $x\in (0,\infty)$, then the function
$$x\mapsto  \ln\frac{\displaystyle\int_0^xM(t)e^{-\alpha t} dt}
{\displaystyle\int_0^x e^{-\alpha t} dt}$$
is also concave in $ \ln x$ for $x\in (0,\infty)$.
\end{thm}

Note that for any nonnegative integer $k$ the classical integral mean of $z^k$ is both   logarithmic  convex and   logarithmic  concave. So, we obtain the following corollary, which is part (i) of \cite[Theorem 7]{WX}.

\begin{cor}\label{c3}
Let $(\alpha,p)\in [0,\infty)\times(0,\infty)$. If $k$ is a nonnegative integer, then the function
$r\mapsto \ln  \mathsf{M}_{p,\alpha}(z^k,r)$ is concave in $ \ln r$ for $r\in (0,\infty)$.
\end{cor}

\medskip
{\it Notation}. In the forthcoming sections, we will employ the symbol $\equiv$ when a new notation is introduced, but also use the notation $U\sim V$ when $U$ and $V$ have the same sign.

\section{Four Lemmas}\label{s2}

This section collects four lemmas which will be used in the proofs of Theorems \ref{t1}-\ref{t2}-\ref{t3}.

The first two lemmas come from \cite{WX}
\begin{lem}\label{2}
Suppose $f$ is positive and twice differentiable on $(0,\infty)$. Then
\begin{enumerate}
\item[(a)] $f(x)$ is convex in $ \ln x$ if and only if $f(x^2)$ is convex in $ \ln x$ and $f(x)$ is concave in $ \ln x$ if and only if $f(x^2)$ is concave in $ \ln x$.
\item[(b)] Let
$$D(f(x))\equiv\frac{f'(x)}{f(x)}
+x\frac{f''(x)}{f(x)}-x\left(\frac{f'(x)}{f(x)}\right)^2.$$ Then $ \ln f(x)$ is convex in $ \ln x$ if and only if
$D(f(x))\ge 0$ and $ \ln f(x)$ is concave in $ \ln x$ if and only if
$D(f(x))\le 0$ for all $x\in(0,\infty)$.
\end{enumerate}
\end{lem}

\begin{lem}\label{3}
Suppose $f=f_1/f_2$ is a quotient of two positive and twice differentiable functions
on $(0,\infty)$. Then
$$D(f(x))=D(f_1(x))-D(f_2(x))$$
for $x\in(0,\infty)$. Consequently, $ \ln f(x)$ is convex in $ \ln x$ if and only if
$$D(f_1(x))-D(f_2(x))\ge 0$$
on $(0,\infty)$ and $ \ln f(x)$ is concave in $ \ln x$ if and only if
$$D(f_1(x))-D(f_2(x))\le 0$$
on $(0,\infty)$.
\end{lem}

We next establish several estimates for the function $\varphi$.

\begin{lem}\label{4}
Suppose
$$
\begin{cases}
\alpha\in\mathbb{R};\\
x\in[0,\infty);\\
\varphi=\varphi(x)\equiv\int_0^re^{-\alpha t^2}\,2tdt=\int_0^xe^{-\alpha t}\,dt=\frac{1}{-\alpha}\left(e^{-\alpha x}-1\right).
\end{cases}
$$
Then
\begin{enumerate}
\item[(a)] $1-\alpha \varphi(x)=\varphi'(x)$.
\item[(b)] $\varphi(x)-x\ge0$ when $\alpha\le0$ and $\varphi(x)-x\le0$ when $\alpha\ge0$.
\item[(c)] $g_1(x)\equiv x(1-\alpha x)-\varphi(x)\le0$ when $\alpha\ge0$.
\item[(d)] $g_2(x)\equiv \alpha\varphi^2(x)-2(1+\alpha x)\varphi(x)+2x\le0$.
\item[(e)] $g_3(x)\equiv x-(1+\alpha x)\varphi(x)$ is nonnegative when $\alpha\le0$ and not positive when $\alpha\ge0$.
\end{enumerate}
\end{lem}

\begin{proof}
Part (a) follows from the fact that
$$\varphi(x)=\frac{1}{-\alpha}\left(e^{-\alpha x}-1\right),\quad \varphi'(x)=e^{-\alpha x}.$$

Part (b) follows from the fact that $e^{-\alpha x}\ge1$ for $\alpha\le0$ and $x\in[0,\infty)$.

A direct computation shows that
$$g_1'(x)=1-2\alpha x-\varphi'(x)$$
and $$g_1''(x)=-2\alpha-\varphi''(x)=\alpha(\varphi'(x)-2).$$
It follows that
$g_1''(x)\le0$ when $\alpha\ge0$. So we have $g_1'(x)\le g_1'(0)=0$ and $g_1(x)\le g_1(0)=0$ for all $x\in [0,\infty)$. This proves (c).

Another computation gives
\begin{eqnarray*}
g_2'(x)&=&2\alpha\varphi\varphi'-2\alpha\varphi-2(1+\alpha x)\varphi'+2\\
&=&2\alpha\varphi\varphi'-2(1+\alpha x)\varphi'+2\varphi'\\
&=&2\alpha(\varphi-x)\varphi'.
\end{eqnarray*}
By part (b), we have $g_2'(x)\le0$ for all $x\in[0,\infty)$. Therefore, $g_2(x)\le g_2(0)=0$ for
all $x\in [0,\infty)$. This proves (d).

A similar computation produces
\begin{eqnarray*}
g_3'(x)&=&1-\alpha\varphi(x)-(1+\alpha x)\varphi'(x)\\
&=&-\alpha x\varphi'(x),
\end{eqnarray*}
which yields $g_3(x)\ge g_3(0)=0$ for all $x\in [0,\infty)$ when $\alpha\le0$ and $g_3(x)\le g_3(0)=0$ for all $x\in [0,\infty)$ when $\alpha\ge0$. This proves (e) and
completes the proof of the lemma.
\end{proof}

Finally, Lemma \ref{4} is applied to derive the following fundamental property.

\begin{lem}
\label{5}
Given a nonconstant entire function $f:\mathbb C\to\mathbb C$, suppose
$$
\begin{cases}
\alpha\in\mathbb{R};\\
x\in[0,\infty);\\
p\in (0,\infty);\\
M(x)\equiv M_p(f,\sqrt{x})=\int_0^{2\pi}|f(\sqrt{x}e^{i\theta})|^p\,d\theta;\\
h=h(x)\equiv\int_0^r M_p(f,t)e^{-\alpha t^2}\,2tdt=\int_0^xM(t)e^{-\alpha t}\,dt.
\end{cases}
$$
Let
$$
\begin{cases}
A=A(x)\equiv\frac{\varphi(x)-x}{\varphi^2(x)};\\
B=B(x)\equiv (1-\alpha x)+x\frac{M'(x)}{M(x)};\\
C=C(x)\equiv x\varphi'(x);\\
\Delta(x)\equiv D(h(x))-D(\varphi(x)).
\end{cases}
$$
Then \begin{enumerate}
\item[(a)]
$S=S(x)=\sqrt{B^2-4AC}>0\ \forall\  x\in(0,\infty).$
\item[(b)] $\Delta(x)\sim-A \frac{h^2}{M^2}+B\frac{h}{M}-C\ \forall\ x\in (0,\infty)$
\end{enumerate}
\end{lem}

\begin{proof} (a) Noticing that $M'>0$ and $M>0$, together with $x-(1+\alpha x)\varphi\ge0$ by Lemma \ref{4}, we have
\begin{eqnarray*}
B^2-4AC&=&\left[(1-\alpha x)+\frac{xM'}M\right]^2-4x\varphi'
\frac{\varphi-x}{\varphi^2}\\&>&(1-\alpha x)^2-4x\varphi'
\frac{\varphi-x}{\varphi^2}\\&=&\frac{(2x-(1+\alpha x)\varphi)^2}{\varphi^2}\\
&>&0.
\end{eqnarray*}

(b) Since
$$\varphi'=e^{-\alpha x},\quad \varphi''=-\alpha e^{-\alpha x},$$
we have
\begin{eqnarray*}
D(\varphi(x))=\frac{(1-\alpha x)\varphi'}{\varphi}-x\frac{(\varphi')^2}{\varphi^2}
=\frac{(\varphi-x)\varphi'}{\varphi^2}.
\end{eqnarray*}
On the other hand,
$$h'=h'(x)=M(x)\varphi',$$
and
$$h''=h''(x)=[M'(x)-\alpha M(x)]\varphi'.$$
It follows from simple calculations that
$$D(h)=\frac{hh'+xhh''-x(h')^2}{h^2}
=\frac{(1-\alpha x)Mh+xM'h-xM^2\varphi'}{h^2}\varphi'.$$
Therefore,
\begin{eqnarray*}
\Delta(x)&=&\frac{\varphi' M}{h^2}(hB-CM)-\varphi'A\\
&\sim&-A \frac{h^2}{M^2}+B\frac{h}{M}-C.
\end{eqnarray*}
\end{proof}

\section{Proofs of Theorems \ref{t1}\&\ref{t2}}

To verify  Theorems \ref{t1}\&\ref{t2}, we use (ii) of Lemma~\ref{2} to show the logarithmic convexity of $h(x)/\varphi(x)$ on
$(0,\infty)$. According to Lemma~\ref{3}, this will be accomplished if we can prove $\Delta(x)\ge 0$.

Suppose that $\alpha<0$. From Lemma~\ref{4} it follows that $A(x)$, $B(x)$ and $C(x)$ are all positive
on $(0,\infty)$ as $\alpha\le0$ and $M'/M>0$. By some direct computations, we have
\begin{eqnarray*}
xA'(x)&=&\frac{x}{\varphi^3}\left[\alpha\varphi^2-2(1+\alpha x)\varphi+2x\right],\\
B'&=&-\alpha+\left(x\frac{M'(x)}{M(x)}\right)',\\
 xC'&=&x(1-\alpha x)\varphi'=(1-\alpha x)C.
 \end{eqnarray*}
Thus, an application of Lemma~\ref{5} yields that $\Delta(x)\ge0$ is equivalent to
\begin{equation}\label{eq4}
-\frac{\sqrt{B^2-4AC}}{2A}\leq \frac{h}{M}-\frac{B}{2A}\leq\frac{\sqrt{B^2-4AC}}{2A}.
\end{equation}

Since the function $M$ is positive and increasing, we have
$$B(x)\ge1-\alpha x\ge0,\quad h(x)\leq\int_0^xM(x)\varphi'(t) dt=M(x)\varphi(x).$$
It follows from this, the proof of Lemma~\ref{5}, part (b) of Lemma~\ref{4}, and the
triangle inequality that
\begin{eqnarray*}\frac{B+\sqrt{B^2-4AC}}{2A}&\ge&\frac{(1-\alpha x)+\left|1+\alpha x-\frac{2x}{\varphi}\right|}{2A}\\
&\ge&\frac{1-\alpha x+1+\alpha x-\frac{2x}{\varphi}}{2A}\\
&=&\frac{2(\varphi-x)}{2A\varphi}=\varphi\ge\frac{h}{M}.
\end{eqnarray*}
This proves the right half of (\ref{eq4}).

To prove the left half of (\ref{eq4}), we write
\begin{equation}\label{eq1}
\delta=\delta(x)=h-M\frac{B-\sqrt{B^2-4AC}}{2A}
\end{equation}
for $x\in(0,\infty)$ and proceed to show that $\delta(x)$ is nonnegative.
It follows from the elementary identity
$$\frac{B-\sqrt{B^2-4AC}}{2A}=\frac{2C}{B+\sqrt{B^2-4AC}}$$
that $\delta(x)\to0$ as $x\to 0^+$. If we can show that $\delta'(x)\ge0$
for all $x\in(0,\infty)$, then we will obtain
$$\delta(x)\ge\lim_{t\to0^+}\delta(t)=0,\quad x\in(0,\infty).$$
The rest of the proof is thus devoted to proving the
inequality $\delta'(x)\ge0$ for $x\in(0,\infty)$.

By a direct computation, we have
\begin{eqnarray*}
\delta'(x)&=&M\varphi'-\frac{M'A-MA'}{2A^2}\left(B-\sqrt{B^2-4AC}\right)\\
&-&\frac{M}{2A}\left(B'-\frac{BB'-2(A'C+AC')}{\sqrt{B^2-4AC}}\right)\\
&=&\frac{M}{x}\left[C-\left(\frac{xM'}{M}-\frac{xA'}{A}\right)\frac{B-S}{2A}\right.\\
&&\left.+\frac{xB'}{2AS}(B-S)-\frac{xA'C+(1-\alpha x)AC}{AS}\right].
\end{eqnarray*}

Noticing that $B+S\ge0$ for any $\alpha\in (-\infty,\infty)$. Multiplying $\frac{xS(B+S)}{MC}$ on both sides of the above expressions of $\delta'(x)$, and then using (\ref{eq4}) and (\ref{eq1}), we obtain that
\begin{eqnarray*}
\delta'&\sim&(B+S)S-2S\left(\frac{xM'}M-\frac{xA'}{A}\right)\\
&&+2xB'-\left(\frac{xA'}{A}+1-\alpha x\right)(B+S)\\
&=&\left(\frac{xM'}M-\frac{xA'}{A}\right)(B-S)+2xB'-4AC\\
&=&\left(\frac{xM'}M-\frac{xA'}{A}\right)(B-S)+2xA'\varphi+2xD(M(x))\equiv d_1.\end{eqnarray*}

We will determine the sign of $d_1$. To this end, we let$$y=\frac{xM'}M$$and
$$d_2=\left(y-\frac{xA'}{A}\right)(B-S)+2xA'\varphi.$$
Note that $A,C,A'$ are independent of $y$ and $B=1-\alpha x+y$. A simple computation shows that $S'(y)=B/S$. Multiplying $$\frac{B+S}{-2xA'\varphi}$$ on the both sides of the above expressions of $d_2$, we obtain that
\begin{eqnarray*}
d_2&\sim&\left(y-\frac{xA'}{A}\right)\cdot \frac{4AC}{-2xA'\varphi}-(B+S)\\
&=&\frac{\alpha}{\varphi A'}y+\left(\frac{2x}{\varphi}-1-\alpha x\right)-S\equiv d_3.\end{eqnarray*}

\begin{proof}[Proof of Theorem \ref{t1}, Continued] Since $\alpha<0$, using the assumption $D(M(x))\ge0$ we have $d_1\ge d_2$. By using Lemma \ref{4}(e) we can easily see that $$\frac{\alpha}{\varphi A'}\ge1.$$ It follows from a direct computation and Lemma \ref{4} that
\begin{eqnarray*}
d_2&\sim&\left[\frac{\alpha}{\varphi A'}y+\left(\frac{2x}{\varphi}-1-\alpha x\right)\right]^2-S^2\\
&=&y\left[\left(\frac{\alpha^2}{(\varphi A')^2}-1\right)y+2\frac{\alpha}{\varphi A'}\left(\frac{2x}{\varphi}-1-\alpha x\right)-2(1-\alpha x)\right]\\
&\sim&y-y_0,\end{eqnarray*}where$$y_0=\frac{[x(1-\alpha x)-\varphi][\alpha\varphi^2-2(1+\alpha x)\varphi+2x]}{(\varphi-x)[x-(1+\alpha x)\varphi]}.$$

Since $$-\alpha\left[\varphi-x(1-\alpha x)\right]=e^{-\alpha x}-1+\alpha x-\alpha^2 x^2,$$we consider $$u(t)=e^{t}-1-t-t^2,\qquad t>0.$$ It follows from elementary calculus that $u(t)$ has a unique root
$t_0=1.79\cdots$ on $(0,\infty)$ and $u(t)<0$ on $(0,t_0)$ and $u(t)>0$ on $(t_0,\infty)$. Hence $\varphi-x(1-\alpha x)$ has a unique root $x_0=t_0/(-\alpha)$ on $(0,\infty)$ and $\varphi-x(1-\alpha x)<0$ on $(0,x_0)$ and $\varphi-x(1-\alpha x)>0$ on $(x_0,\infty)$.

Note that $y_0\sim\varphi-x(1-\alpha x)$. When $x\le x_0$, $y_0\le0$, so we have $d_2\ge0$ and hence $d_1\ge0$ on $(0,x_0)$. In particular, the function
$$x\mapsto  \ln\frac{\displaystyle\int_0^xM(t)e^{-\alpha t} dt}
{\displaystyle\int_0^xe^{-\alpha t} dt}$$
is convex in $ \ln x$ for $x\in (0,x_0)$. As for $x\in I\cap(x_0,\infty)$, we have $y_0\ge0$, the assumption $y\ge y_0$ when $x\in I$ implies $d_2\ge0$ and hence $d_1\ge0$ on $I\cap(x_0,\infty)$. This shows that $d_1$ is always nonnegative and completes the proof of Theorem \ref{t1}.
\end{proof}

\begin{proof}[Proof of Theorem \ref{t2}, Continued] Since $\alpha<0$, it follows from Lemma \ref{4} that 
$$
B-S\sim B^2-S^2=4AC\ge0\quad\&\quad \frac{xA'}{A}\le0.
$$
Hence\begin{eqnarray*}d_2'(y)&=&B-S+\left(y-\frac{xA'}{A}\right)(1-\frac{B}{S})\\&\sim& S-\left(y-\frac{xA'}{A}\right)\\&\sim& S^2-\left(y-\frac{xA'}{A}\right)^2\\
&=&\frac{2(\varphi^2-x(3+\alpha x)\varphi+2x^2)}{\varphi(\varphi-x)}y\\&&-\frac{(\varphi-x(1-\alpha x))(-(1+2\alpha x)\varphi^2+x(5+3\alpha x)\varphi-4x^2)}{\varphi(\varphi-x)^2}.
\end{eqnarray*}
It follows from Lemma \ref{4} that 
$$\varphi-x\ge0\quad\&\quad \varphi^2-x(3+\alpha x)\varphi+2x^2=(\varphi-x)^2+x(x-(1+\alpha x)\varphi)\ge0,$$
and 
\begin{eqnarray*}
&&-(1+2\alpha x)\varphi^2+x(5+3\alpha x)\varphi-4x^2\\
&&\qquad=\ x(\varphi-x)+(x-(1+\alpha x)\varphi)(\varphi-x)-x[\alpha\varphi^2-2(1+\alpha x)\varphi+2x]\\
&&\qquad\ge\ 0.
\end{eqnarray*}
So we have $d_2'(y)\sim y-y^*$,
where\begin{eqnarray*}y^*=\frac{(\varphi-x(1-\alpha x))(-(1+2\alpha x)\varphi^2+x(5+3\alpha x)\varphi-4x^2)}{2(\varphi-x)(\varphi^2-x(3+\alpha x)\varphi+2x^2)}.\end{eqnarray*}
Note that $y^*\sim\varphi-x(1-\alpha x)$.

When $x\le x_0$, $y^*\le0$, $d_2'(y)$ and hence $d_2$ is nonnegative. As for $x>x_0$, $y^*\ge0$, $d_2(y)$ attains its minimum value at $y^*\in(0,\infty)$. A direct computation shows that$$d_2(y^*)=-\frac{1}{2}\left(1+\frac{\alpha x^2}{\varphi-x}\right)^2.$$So we have \begin{eqnarray*}d_1&\ge&2xD(M(x))+d_2(y^*)\\
&\sim& D(M(x))-\frac{1}{4x}\left(1+\frac{\alpha x^2}{\varphi-x}\right)^2.\end{eqnarray*}
This shows that $d_1$ is always nonnegative and completes the proof of Theorem~\ref{t2}.
\end{proof}

\section{Proof of Theorem \ref{t3}}

To demonstrate Theorem \ref{t3}, we indicate how to adapt the proof of
Theorem \ref{t1} or Theorem \ref{t2} above to show that $\Delta(x)\le0$.

Suppose $\alpha>0$. Then $A<0$ by Lemma \ref{4} and so $\Delta(x)\le 0$ is equivalent to $$-\frac{\sqrt{B^2-4AC}}{2A}\leq \frac{h}{M}-\frac{B}{2A}.$$So we need only to prove that $\delta=\delta(x)$ defined in (\ref{eq1}) is not positive for all $x\in(0,\infty)$. It is enough for us to prove that $\delta'(x)\le0$ since $\delta(0)=0$. We have proved that $\delta'(x)\sim d_1$. Since $M(x)$ is   logarithmic  concave, that is, $D(M(x))\le0$, we obtain $d_1\le d_2$. But $d_2\sim d_3$. Recall that
$$
d_3=\frac{\alpha}{\varphi A'}y+\left(\frac{2x}{\varphi}-1-\alpha x\right)-S.$$
Noticing that
$$\frac{\alpha}{\varphi A'}\le0
$$
by Lemma \ref{4}, we have
$$d_3\le\left(\frac{2x}{\varphi}-1-\alpha x\right)-S.$$
By the proof of Lemma \ref{5}, we have $$S\ge\left|\frac{2x}{\varphi}-1-\alpha x\right|.$$ Thus we get $d_3\le0$ and hence $d_2\le0$. This shows that $d_1\le0$ and completes the proof of Theorem \ref{t3}.

\end{document}